\documentclass[twoside,
]{amsart}

  \usepackage[utf8]{inputenc}
  \usepackage{amsmath}
  \usepackage{amssymb}
  \usepackage{amsthm}
  \usepackage{mathrsfs}
  \usepackage{bm}
  \usepackage{dsfont}
  \usepackage{pifont}
  \usepackage{graphicx}
  \usepackage{xmpmulti}

  \usepackage{hyperref}
  
  \usepackage[svgnames,x11names]{xcolor} 

  \hypersetup{
    colorlinks, linkcolor={Chocolate4},
    citecolor={Chocolate4}, urlcolor={DarkOliveGreen4}
  }

  \usepackage{color}
  
  \usepackage{todonotes}

	\newtheoremstyle{slanted}
	{}
	{}
	{\slshape}
	{}
	{\bfseries}
	{.}
	{ }
	{}
	
	\theoremstyle{slanted}
	\newtheorem{theo}{Theorem}[section]
	\newtheorem*{claim*}{Claim}
	\newtheorem{prop}[theo]{Proposition}

	\newtheorem{lemma}[theo]{Lemma}
	\newtheorem{definition}[theo]{Definition}
	\newtheorem{corollary}[theo]{Corollary}
	\newtheorem{remark}[theo]{Remark}

	\newcommand{\egdef}{:=}
	\DeclareMathOperator{\Id}{Id}	
	\newcommand{\tend}[3][]{\xrightarrow[#2\to#3]{#1}}

	\newcommand{\ind}[1]{\mathds{1}_{#1}} 
	\newcommand{\ZZ}{\mathbb{Z}}

	\newcommand{\RR}{\mathbb{R}}
	
	\newcommand{\NN}{\mathbb{N}}
	\newcommand{\A}{\mathscr{A}}
	\newcommand{\B}{\mathscr{B}}
	\newcommand{\D}{\mathscr{D}}

	\newcommand{\cd}{C^{d}}
	\newcommand{\td}{T^{\times d}}
	\newcommand{\typx}{\overline{x}}
	
	\newcommand{\stack}[2]{\array{c}{\scriptstyle #1}\\[-1.1ex]{\scriptstyle #2}\endarray}

\title[Cartesian powers of infinite Chacon]{Invariant measures for Cartesian powers of Chacon infinite transformation}
\author{\'{E}lise Janvresse, Emmanuel Roy and Thierry de la Rue}
\address{\'Elise Janvresse: 
Laboratoire Amiénois de Mathématiques Fondamentales et Appliquées, CNRS-UMR 7352, Université de Picardie Jules Verne, 33 rue Saint Leu, F80039 Amiens cedex 1,
France.}
\email{Elise.Janvresse@u-picardie.fr}
\address{Emmanuel Roy: Laboratoire Analyse, Géométrie et Applications, Université Paris 13 Institut Galilée,
99 avenue Jean-Baptiste Clément
F93430 Villetaneuse, France.}
\email{roy@math.univ-paris13.fr}
\address{Thierry de la Rue:
Laboratoire de Mathématiques Rapha\"el Salem,
Université de Rouen, CNRS,
Avenue de l'Université,
F76801 Saint \'Etienne du Rouvray, France.}
\email{Thierry.de-la-Rue@univ-rouen.fr}
\thanks{Research partially supported by French research group GeoSto
(CNRS-GDR3477)}

\begin{document}
\bibliographystyle{amsplain}

\maketitle
\begin{abstract}  
We describe all boundedly finite measures which are invariant by Cartesian powers of an infinite measure preserving version of Chacon transformation. 
All such ergodic measures are products of so-called \emph{diagonal measures}, which are measures generalizing in some way the measures supported on a graph.
Unlike what happens in the finite-measure case, this class of diagonal measures is not reduced to measures supported on a graph arising from powers of the transformation: 
it also contains some weird invariant measures, whose marginals are singular with respect to the measure invariant by the transformation. 
We derive from these results that the infinite Chacon transformation has trivial centralizer, and has no nontrivial factor. 

At the end of the paper, we prove a result of independent interest, providing sufficient conditions for an infinite measure preserving dynamical system defined on a Cartesian product to decompose into a direct product of two dynamical systems.
\end{abstract}



\section{Chacon infinite transformation}
\subsection{Introduction}
The classical Chacon transformation, which is a particular case of a finite measure preserving rank-one transformation, is considered as one of the jewels of ergodic theory~\cite{kt2006}. It has been formally described in~\cite{Friedman}, following ideas introduced by Chacon in 1966. Among other properties, it has been proved to have no non trivial factor, and to commute only with its powers~\cite{DJ1978}. More generally, it has minimal self-joinings~\cite{DJRS1980}. For a symbolic version of this transformation, Del~Junco and Keane~\cite{DJK1985} have also shown that if $x$ and $y$ are not on the same orbit, and at least one of them is outside a countable set of exceptional points, then $(x,y)$ is generic for the product measure. 

Adams, Friedman and Silva introduced in 1997 (\cite{AFS1997}, Section~2) an infinite measure preserving rank-one transformation which can be seen as the analog of the classical Chacon transformation in infinite measure. They proved that all its Cartesian powers are conservative and ergodic. 

This transformation, denoted by $T$ throughout the paper, is the main object of the present work. We recall its construction on $\RR_+$ by cutting and stacking in the next section. In particular, we are interested in lifting known results about self-joinings of Chacon transformation to the infinite-measure case. This leads us to study all ergodic measures on $(\RR_+)^d$ which are boundedly finite and $\td$-invariant: we prove in Theorem~\ref{thm:msj} that all such measures are products of so-called \emph{diagonal measures}, which are measures generalizing in some way the measures supported on a graph (see Definition~\ref{def:diagonal}). These diagonal measures are studied in details in Section~\ref{sec:diagonal}. Surprisingly, besides measures supported on a graph arising from powers of $T$, we prove the existence of some weird invariant measures whose marginals are singular with respect to the Lebesgue measure. (It may happen that these marginals take only the values 0 or $\infty$, which is for example the 
case for a product measure. But even in such a case, it makes sense to consider their absolute continuity.)

However, in Section~\ref{sec:joinings}, we prove in Proposition~\ref{prop:alphamu} that these weird measures cannot appear in the ergodic decomposition of selfjoinings of $T$. These selfjoinings are therefore convex combinations of graph measures arising from powers of $T$. This allows to obtain the expected consequences that the infinite Chacon transformation has trivial centralizer, and has no nontrivial $\sigma$-finite factor.

At the end of the paper, we prove in Annex~A a result used in the proof of Theorem~\ref{thm:msj} which can be of independent interest: Theorem~\ref{thm:product} provides sufficient conditions for an infinite measure preserving dynamical system defined on a Cartesian product to decompose into a direct product of two dynamical systems.

\medskip 
The authors thank Alexandre Danilenko for having pointed out a mistake in an earlier version of this paper. 

\subsection{Construction of Chacon infinite transformation}

We define the transformation on $X\egdef\RR_+$: In the first step we consider the interval $[0,1)$, which is cut into three subintervals of equal length. We take the extra interval $[1,4/3)$ and stack it above the middle piece, and 4 other extra intervals of length $1/3$ which we stack above the rightmost piece. Then we stack all intervals left under right, getting a tower of height $h_1=8$. The transformation $T$ maps each point to the point exactly above it in the tower. At this step $T$ is yet undefined on the top level of the tower.

After step $n$ we have a tower of height $h_n$, called tower~$n$, made of intervals of length $1/3^n$ which are closed to the left and open to the right. At step~$(n+1)$, tower~$n$ is cut into three subcolumns of equal width. We add an extra interval of length $1/3^{n+1}$ above the middle subcolumn and $3h_n+1$ other extra intervals above the third one. We pick the extra intervals successively by taking the leftmost interval of desired length in the unused part of $\RR_+$.  Then we stack the three subcolumns left under right and get tower~$n+1$ of height $h_{n+1}=2(3h_n+1)$.

Extra intervals used at step $n+1$ are called \emph{$(n+1)$-spacers}, so that tower~$(n+1)$ is the union of tower~$n$ with $3h_{n}+2$ such $(n+1)$-spacers. The total measure of the added spacers being infinite, we get at the end a transformation $T$ defined on $\RR_+$, which preserves the Lebesgue measure $\mu$.

For each $n\ge1$, we define $C_n$ as the bottom half of tower~$n$: $C_n$ is the union of $h_n/2$ intervals of width $1/3^n$, which contains the whole tower~$(n-1)$. Notice that $C_n\subset C_{n+1}$, and that $X=\bigcup_n C_n$.
We also define a function $t_n$ on tower~$n$, taking values in $\{1,2,3\}$, which indicates for each point whether it belongs to the first, the second, or the third subcolumn of tower~$n$.

\begin{figure}[htp]
  \centering
  \includegraphics{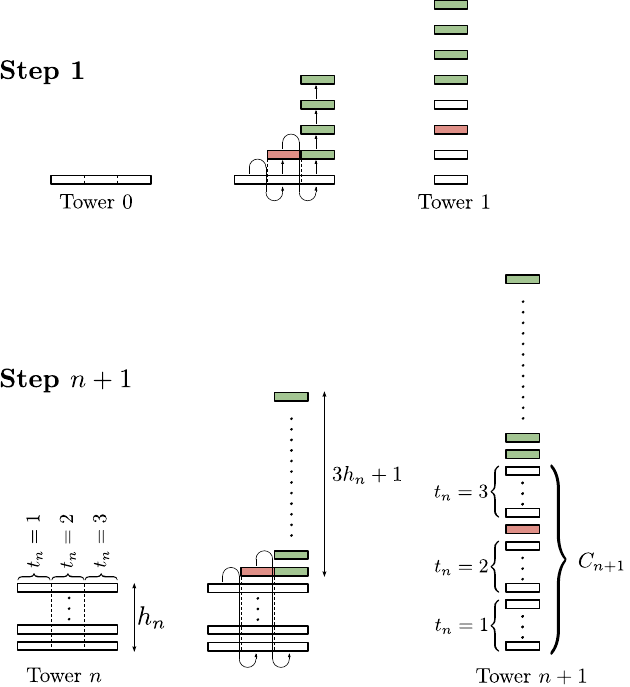}
  \caption{Construction of Chacon infinite measure preserving transformation by cutting and stacking}
  \label{fig:construction}
\end{figure}

\section{Ergodic invariant measures for Cartesian powers of the infinite Chacon transformation}
Let $d\ge1$ be an integer. We consider the $d$-th Cartesian power of the transformation $T$:
 \[\td:X^d\ni(x_1,\ldots,x_d)\mapsto(Tx_1,\ldots,Tx_d).\]
\begin{definition}
\label{def:locally-finite}
A measure $\sigma$ on $X^d$ is said to be \emph{boundedly finite} if $\sigma(A)<\infty$ for all bounded measurable subset $A\subset X^d$.
\end{definition}
Equivalently, $\sigma$ is boundedly finite if $\sigma(\cd_n)<\infty$ for each $n$. Obviously, boundedly finite implies $\sigma$-finite.

\subsection{Products of diagonal measures}

Our purpose in this section is to describe, for each $d\ge1$, all boundedly finite measures on $X^d$ which are ergodic for the action of $\td$. 
Examples of such measures are given by so-called \emph{graph joinings}:
A measure $\sigma$ on $X^d$ is called a graph joining if there exist some real $\alpha>0$ and $(d-1)$ $\mu$-preserving transformations $S_2,\ldots,S_d$, commuting with $T$, and such that 
\[
  \sigma(A_1\times\cdots\times A_d)  = \alpha \mu (A_1\cap S_2^{-1}(A_2)\cap\cdots\cap S_d^{-1}(A_d)).
\]
In other words, $\sigma$ is the pushforward measure of $\mu$ by the map $x\mapsto(x,S_2x,\ldots,S_dx)$. In the case where the transformations $S_j$ are powers of $T$, such a graph joining is a particular case of what we call a \emph{diagonal measure}, which we define now.

From the properties of the sets $C_n$, it follows that $\cd_n\subset\cd_{n+1}$, and that $X^d=\bigcup_n \cd_n$. We call \emph{$n$-box} a subset of $X^d$ which is a Cartesian product $I_1\times \cdots\times I_d$, where each $I_j$ is a level of $C_n$. We call \emph{$n$-diagonal} a finite family of $n$-boxes of the form 
\[
  B, \td B,\ldots, (\td)^{\ell}B,
\]
which is maximal in the following sense: $ (\td)^{-1}B\not\subset \cd_n$ and $ (\td)^{\ell+1}B\not\subset \cd_n$.

\begin{definition}
\label{def:diagonal}
A boundedly finite, $\td$-invariant measure $\sigma$ on $X^d$ is said to be a \emph{diagonal measure} if there exists an integer $n_0$ such that, for all $n\ge n_0$, $\sigma|_{\cd_n}$ is concentrated on a single $n$-diagonal.
\end{definition}

Note that, for $d=1$, there is only one $n$-diagonal for any $n$, therefore $\mu$ is itself a 1-dimensional diagonal measure. A detailed study of diagonal measures will be presented in Section~\ref{sec:diagonal}.

\begin{theo}
  \label{thm:msj}
  Let $d\ge1$, and let $\sigma$ be a nonzero, $\td$-invariant, boundedly finite measure on $X^d$, such that the system $(X^d,\sigma,\td)$ is ergodic.  Then there exists a partition of $\{1,\ldots,d\}$ into $r$ subsets $I_1,\ldots,I_r$, such that $\sigma=\sigma^{I_1}\otimes \cdots\otimes \sigma^{I_r}$, where $\sigma^{I_j}$ is a diagonal measure on $X^{I_j}$. 
  
  If the system $(X^d,\sigma,\td)$ is totally dissipative, $\sigma$ is a diagonal measure supported on a single orbit.
\end{theo}

We will prove the theorem by induction on $d$. The following proposition deals with the case $d=1$.

\begin{prop}
  \label{prop:d=1}
  The Lebesgue measure $\mu$ is, up to a multiplicative constant, the only $T$-invariant, boundedly finite measure on $X$.
\end{prop}
\begin{proof}
  Let $\sigma$ be a $T$-invariant $\sigma$-finite measure. Then for each $n$, the intervals which are levels of tower~$n$ have the same measure. Since the successive towers exhaust $\RR_+$, we get that for each $n$,  all intervals of the form $[j/3^n,(j+1)/3^n)$ for integers $j\ge 0$ have the same measure $\sigma_n$. Obviously $\sigma_{n+1}=\sigma_n/3$. Since $\sigma$ is boundedly finite, $\sigma_0<\infty$. Hence $\sigma_n<\infty$ and $\sigma$ is, up to the multiplicative constant $\sigma_0$, equal to the Lebesgue measure.
\end{proof}

Observe that assuming only $\sigma$-finiteness for the measure $\sigma$ is not enough: The counting measure on rational points is $\sigma$-finite, $T$-invariant, but singular with respect to Lebesgue measure. Can we have a counterexample where $\sigma$ is conservative?

\subsection{Technical lemmas}
 
 In the following, $d$ is an integer, $d\ge2$.

\begin{lemma}
  \label{lemma:BtoSB}
  Let $G_1\sqcup G_2=\{1,\ldots,d\}$ be a partition of $\{1,\ldots,d\}$ into two disjoint sets, one of which is possibly empty. Let us define a transformation $S:X^d\to X^d$  by 
  \[
      S(y_1,\ldots, y_d) \egdef (z_1,\ldots,z_d),\text{ where } z_i\egdef\begin{cases}
                                                                 T y_i &\text{ if }i\in G_1, \\
                                                                 y_i &\text{ if }i\in G_2.
                                                               \end{cases}
   \]
   Let $n\ge1$, let $B$ be an $n$-box, 
   and let $x=(x_1,\ldots,x_d)\in \cd_n$. If $t_n(x_i)=1$ for $i\in G_1$ and $t_n(x_i)=2$ for $i\in G_2$, then 
   \[
     x\in B\Longleftrightarrow (\td)^{h_n+1}x\in SB.
   \]
Similarly, if $t_n(x_i)=2$ for $i\in G_1$ and $t_n(x_i)=3$ for $i\in G_2$, then 
   \[
     x\in SB\Longleftrightarrow (\td)^{-h_n-1}x\in B.
   \]
\end{lemma}

\begin{proof}
 Let $x=(x_1,\ldots,x_d)\in \cd_n$ such that $t_n(x_i)=1$ for $i\in G_1$ and $t_n(x_i)=2$ for $i\in G_2$. 
For each $1\le i\le d$, let $L_i$ be the level of $C_n$ containing $x_i$.
 If $i\in G_1$, $T^j x_i$, $j$ ranging from $1$ to $h_n+1$, never goes through an $(n+1)$-spacer, hence $T^{h_n+1}x_i\in T L_i$ (see Figure~\ref{fig:construction}).  If $i\in G_2$, $T^j x_i$, $j$ ranging from $1$ to $h_n+1$, goes through exactly one $(n+1)$-spacer, hence $T^{h_n+1}x_i\in L_i$. Hence, $(\td)^{h_n+1}x\in S(L_1\times\cdots\times L_d)$. Observe that, since $B$ is an $n$-box, $B\subset\cd_n$, thus both $B$ and $SB$ are Cartesian products of levels of tower~$n$. 
 We then get
 \begin{align*}
   x\in B & \Longleftrightarrow B=L_1\times\cdots\times L_d\\
   & \Longleftrightarrow SB= S(L_1\times\cdots\times L_d)\\
   & \Longleftrightarrow  (\td)^{h_n+1}x\in SB.
 \end{align*}

 The case $t_n(x_i)=2$ for $i\in G_1$ and $t_n(x_i)=3$ for $i\in G_2$ is handled in the same way. 
\end{proof}

\begin{lemma}
  \label{lemma:tn}
  Let $n\ge2$, $x=(x_1,\ldots,x_d)\in \cd_{n-1}$ and $\ell\ge n$. If $t_\ell(x_i)\in\{1,2\}$ for each $1\le i\le d$, then $(\td)^{h_\ell+1}x\in \cd_n$. 
\end{lemma}
\begin{proof}
Let $B_\ell$ (respectively $B_n$) be the $\ell$-box (respectively the $n$-box) containing $x$. Observe that $B_\ell\subset B_n\subset \cd_{n-1}$ because $x\in\cd_{n-1}$. Applying Lemma~\ref{lemma:BtoSB}, we get $(\td)^{h_\ell+1}x\in SB_\ell\subset SB_n$, where $S$ is the transformation of $X^d$ acting as $T$ on coordinates $i$ such that $t_\ell(x_i)=1$ and acting as $\Id$ on other coordinates.  Since $B_n\subset \cd_{n-1}$, $SB_n\subset\cd_n$, which ends the proof.
\end{proof}

\begin{definition}
Let $x=(x_1,\ldots,x_d)\in X^d$. For each integer $n\ge1$, we call \emph{$n$-crossing for $x$} a maximal finite set of consecutive integers $j\in\ZZ$ such that $(\td)^j x\in \cd_n$. 
\end{definition}

Note that, when $j$ ranges over an $n$-crossing for $x$, $(\td)^j\ x$ successively belongs to the $n$-boxes constituting an $n$-diagonal, and that for each $1\le i\le d$, $t_n(T^jx_i)$ remains constant.

\begin{lemma}
  \label{lemma:separated}
  An $n$-crossing contains at most $h_n/2$ elements.
  Two distinct $n$-crossings for the same $x$ are separated by at least $h_n/2$ integers.
\end{lemma}

\begin{proof}
  The first assertion is obvious since $C_n$ is a tower of height $h_n/2$. Consider the maximum element $j$ of an $n$-crossing for $x=(x_1,\ldots,x_d)$. Then there exists $1\le i\le d$ such that $T^{j}(x_i)\in C_n$, but $T^{j+1}(x_i)\notin C_n$. By construction, $T^{j+\ell}(x_i)\notin C_n$ for all $1\le \ell\le h_n/2$, hence $(\td)^{j+\ell}x\notin \cd_n$.
\end{proof}

\begin{lemma}
  \label{lemma:long-crossing}
Let $j\ge 0$ and $n\ge 2$ such that $(\td)^{j}x\in \cd_{n-1}$. Then $j,j+1,\ldots,j+h_{n-1}/2$ belong to the same $n$-crossing.
\end{lemma}
\begin{proof}
  For all $1\le i\le d$,  $T^{j}(x_i)\in C_{n-1}$, hence for all $1\le \ell\le h_{n-1}/2$, $T^{j+\ell}(x_i)$ belongs to tower~$(n-1)$, hence to $C_n$.
\end{proof}

For $x\in X^d$, let us define $n(x)$ as the smallest integer $n\ge1$ such that $x\in \cd_{n}$. Observe that $x\in\cd_n$ for each $n\ge n(x)$. In particular, for each $n\ge n(x)$, 0 belongs to an $n$-crossing for $x$, which we call the \emph{first $n$-crossing for $x$}. Observe also that the first $(n+1)$-crossing for $x$ contains the first $n$-crossing for $x$. Since $n$-crossings for $x$ are naturally ordered, we refer to the next $n$-crossing for $x$ after the first one (if it exists) as the \emph{second $n$-crossing for $x$}.

\begin{lemma}
  \label{lemma:special_n}
  Let $x=(x_1,\ldots,x_d)\in X^d$ such that, for any $n\ge n(x)$, there exist infinitely many $n$-crossings for $x$ contained in $\ZZ_+$. Then there exist infinitely many integers $n\ge n(x)+1$ such that the first $(n+1)$-crossing for $x$ also contains the second $n$-crossing for $x$. Moreover, for such an integer $n$, $t_n(x_i)\in\{1,2\}$ for each $i\in\{1,\ldots,d\}$, and for $j$ in the second $n$-crossing, we have $t_n(T^jx_i)=t_n(x_i)+1$. 
  \end{lemma}
    
\begin{proof}
Let $m\ge n(x)+1$, and let $\{s,s+1,\ldots,s+r\}$ be the second $m$-crossing for $x$. Define $n\ge m$ as the smallest integer such that $(\td)^{j}x\in\cd_{n+1}$ for each $0\le j\le s+r$. Then the $n$-crossing for $x$ containing zero is distinct from the $n$-crossing for $x$ containing $s$, and these two $n$-crossings are contained in the same $(n+1)$-crossing for $x$. Therefore the first $(n+1)$-crossing for $x$ contains both the first and the second $n$-crossings for $x$.

By Lemma~\ref{lemma:separated}, the first and the second $n$-crossings are separated by at least $h_n/2$, hence each coordinate has to leave $C_n$ between them. 
If we had $t_n(x_i)=3$ for some $i$, then $T^j(x_i)$ would also leave $C_{n+1}$ before coming back to $C_n$, which contradicts the fact that both $n$-crossings are in the same $(n+1)$-crossing. Hence $t_n(x_i)\in\{1,2\}$ for each $i$. 
Moreover, recall that $n\ge m \ge n(x)+1$, thus $x\in \cd_{n-1}$. Hence $x$ satisfies the assumptions of Lemma~\ref{lemma:tn}, with $\ell=n$. Therefore, $(\td)^{h_n+1}x\in \cd_n$, which proves that $h_n+1$ belongs to the second $n$-crossing. At time $h_n+1$, each coordinate has jumped to the following subcolumn: $t_n(T^{h_n+1}x_i)=t_n(x_i)+1$. The conclusion follows as $t_n$ is constant over an $n$-crossing.
\end{proof}

\subsection{Proof of Theorem~\ref{thm:msj}, conservative case}
Now we consider an integer $d\ge2$ such that the statement of Theorem~\ref{thm:msj} (in the conservative case) is valid up to $d-1$. Let $\sigma$ be a nonzero measure on $X^d$, which is boundedly finite, $\td$-invariant, and such that the system $(X^d,\sigma,\td)$ is ergodic and conservative. By Hopf's ergodic theorem, if $A\subset B\subset X^d$ with $0<\sigma(B)<\infty$, we have for $\sigma$-almost every point $x=(x_1,\ldots,x_d)\in X^d$
\begin{equation}
  \label{eq:Hopf}
  \dfrac{\sum_{j\in I}\ind{A}((\td)^jx)}{\sum_{j\in I}\ind{B}((\td)^jx)}
  \tend{|I|}{\infty} \dfrac{\sigma(A)}{\sigma(B)},
\end{equation}
where the sums in the above expression range over an interval $I$ containing 0.

Recall that $\cd_n\subset\cd_{n+1}$, and that $X^d=\bigcup_n \cd_n$. In particular, for $n$ large enough, $\sigma(\cd_n)>0$ (and $\sigma(\cd_n)<\infty$ because $\sigma$ is boundedly finite). By conservativity, this implies that almost every $x\in X^d$ returns infinitely often in $\cd_n$.

We say that $x\in X^d$ is \emph{typical} if, for all $n$ large enough so that $\sigma(\cd_n)>0$,
\begin{itemize}
  \item[(i)] Property~\eqref{eq:Hopf} holds whenever $A$ is an $n$-box and $B$ is $\cd_n$, 
  \item[(ii)] $(\td)^jx\in \cd_n$ for infinitely many integers $j\ge0$.
\end{itemize}
(In fact, it can be shown that (ii) follows from (i), but this requires some work, and we do not need this implication.)
We know that $\sigma$-almost every $x\in X^d$ is typical. Moreover, $\sigma$-almost every $x\in X^d$ satisfies
\begin{equation}
  \label{eq:n(x)}
  \sigma\left(\cd_{n(x)}\right)>0.
\end{equation}
From now on, we consider a fixed typical point $\typx=(\typx_1,\ldots,\typx_d)$ satisfying~\eqref{eq:n(x)}, and we will estimate the measure $\sigma$ along its orbit. 
 By~\eqref{eq:n(x)}, $\typx$ satisfies (ii) for all $n\ge n(\typx)$, thus $\typx$ satisfies the assumption of Lemma~\ref{lemma:special_n}. Hence we are in exactly one of the following two complementary cases.

\smallskip

\noindent {\bf Case 1:} There exists $n_1$ such that, for each $n\ge n_1$ satisfying the condition given in Lemma~\ref{lemma:special_n}, and for each $1\le i\le d$, $t_n(\typx_i)=t_n(\typx_1)$.
\smallskip

\noindent {\bf Case 2:} There exist a partition of $\{1,\ldots,d\}$ into two disjoint nonempty sets \[
                                                     \{1,\ldots,d\}=G_1\sqcup G_2,                                              
                                                                                                 \]
and infinitely many integers $n$ satisfying the condition given in Lemma~\ref{lemma:special_n} such that, for each $i\in G_1$, $t_n(\typx_i)=1$, and for each $i\in G_2$, $t_n(\typx_i)=2$.

\smallskip

Theorem~\ref{thm:msj} will be proved by induction on $d$ once we will have shown the following proposition.

\begin{prop}
  If Case 1 holds, then the measure $\sigma$ is a diagonal measure. 
  
  If Case 2 holds, then $\sigma$ is a product measure of the form 
  \[
    \sigma = \sigma_{G_1}\otimes \sigma_{G_2},
  \]
where, for $i=1,2$, $\sigma_{G_i}$ is a measure on $X^{G_i}$ which is boundedly finite, $T^{\times |G_i|}$-invariant, and such that the system $(X^{G_i},\sigma_{G_i},T^{\times |G_i|})$ is ergodic and conservative.
\end{prop}

\begin{proof}
  All $n$-crossings used in this proof are $n$-crossings for the fixed typical point $\typx$.
  
  First consider Case~1. Let $m\ge n_1$. We claim that every $m$-crossing passes through the same $m$-diagonal as the first $m$-crossing. Let $J\subset \NN$ be an arbitrary $m$-crossing. Define $n$ as the smallest integer $n\ge m$ such that all integers $j\in\{0,\ldots,\sup J\}$ are contained in the same $(n+1)$-crossing. Then $n$ satisfies the conditions of Lemma~\ref{lemma:special_n}: The $(n+1)$-crossing containing 0 contains (at least) two different $n$-crossings, the one containing 0 and the one containing the $m$-crossing $J$.  Since we are in Case~1, all coordinates have met the same number of $(n+1)$-spacers between the $n$-crossing containing 0 and the $n$-crossing containing $J$. Hence the $n$-diagonal where $\typx$ lies is the same as the $n$-diagonal containing $(\td)^{j}\typx$ for $j\in J$. Now we prove the claim by induction on $n-m$. If $n-m=0$ we have the result. Let $k\ge0$ such that the claim is true if $n-m\le k$, and assume that $n-m=k+1$. We consider the $n$-crossing containing 
0: It 
may contain several $m$-crossings, but by the induction hypothesis, all these $m$-crossings correspond to the same $m$-diagonal. 
  Now, we know that the $n$-crossing containing $J$ corresponds to the same $n$-diagonal as the $n$-crossing containing 0, thus all the $m$-crossings it contains correspond to the same $m$-diagonal as the $m$-crossing containing 0. Now, since we have chosen $\typx$ typical, it follows that the $m$-diagonal containing $\typx$ is the only one which is charged by $\sigma$. But this is true for all $m$ large enough, hence $\sigma$ is a diagonal measure.
  
  \smallskip
  
  Let us turn now to Case~2. Consider the transformation $S:X^d\to X^d$ defined as in Lemma~\ref{lemma:BtoSB} by
  \[
     S(y_1,\ldots, y_d) = (z_1,\ldots,z_d),\text{ where } z_i\egdef\begin{cases}
                                                                T y_i \text{ if }i\in G_1, \\
                                                                y_i \text{ if }i\in G_2.
                                                              \end{cases}
  \]
 Let us fix $m$ large enough so that $\sigma(\cd_{m-1})>0$. For each $m$-box $B$, denote by $n_B$ (respectively $n'_B$) the number of times the orbit of $\typx$ falls into $B$ along the first $n$-crossing (respectively the second).  We claim that there exists an $m$-box $B$ such that $SB$ is still an $m$-box, and $\sigma(B)>0$. Indeed, it is enough to take any $m$-box in $\cd_{m-1}$ with positive measure. For such an $m$-box $B$, we want now to compare $\sigma(B)$ and $\sigma(SB)$.
 
  Let $n>m$ be a large integer satisfying the condition stated in Case~2. 
  Partition the $m$-box $B$ into $n$-boxes: since $SB$ is also an $m$-box, for each $n$-box $B'\subset B$, $SB'$ is an $n$-box contained in $SB$, and we get in this way all $n$-boxes contained in $SB$. Let us fix such an $n$-box, and apply Lemma~\ref{lemma:BtoSB}: For each $j$ in the first $n$-crossing, we have
  \[
    (\td)^j\typx\in B' \Longleftrightarrow (\td)^{j+h_n+1}\typx\in SB',
  \]
and in this case, by Lemma~\ref{lemma:separated}, $j+h_n+1$ belongs to the second $n$-crossing. In the same way, for each $j$ in the second $n$-crossing, we have
  \[
    (\td)^j\typx\in SB' \Longleftrightarrow (\td)^{j-h_n-1}\typx\in B',
  \]
and in this case, by Lemma~\ref{lemma:separated}, $j-h_n-1$ belongs to the first $n$-crossing. Summing over all $n$-boxes $B'$ contained in $B$, It follows that
   \begin{equation}
    \label{eq:n_et_n'}
    \text{if both $B$ and $SB$ are $m$-boxes, }n'_{SB}=n_B.
  \end{equation} 
  Set
  \[
    N\egdef \sum_{B} n_B,\quad\text{and}\quad N'\egdef \sum_{B} n'_B,
  \]
  where the two sums range over all $m$-boxes $B$. Since we have chosen $\typx$ typical, and since the length of the first $n$-crossing go to $\infty$ as $n\to\infty$, we can apply~\eqref{eq:Hopf} and get, for any $m$-box $B$, as $n\to\infty$
  \begin{equation}
    \label{eq:Hopf-bis}
    \frac{n_B}{N}=\frac{\sigma(B)}{\sigma(\cd_{m})} + o(1),\quad\text{and}\quad\frac{n_B+n'_B}{N+N'}=\frac{\sigma(B)}{\sigma(\cd_{m})} + o(1).
  \end{equation}
  Since $N'\ge\sum n'_{SB}$ where the sum ranges over the set $\B_m$ of all $m$-boxes $B$ such that $SB$ is still an $m$-box, we get by~\eqref{eq:n_et_n'} 
  \[
    N' \ge \sum_{B\in\B_m} n_B.
  \]
Then, applying the left equality in~\eqref{eq:Hopf-bis} for all $B\in\B_m$, we obtain
\[
  \frac{N'}{N} \ge \frac{\sum_{B\in\B_m} \sigma(B)}{\sigma(\cd_{m})}+o(1).
\]
As we know that $\sum_{B\in\B_m} \sigma(B)>0$, it follows that $N'/N$ is larger than some positive constant for $n$ large enough, and we can deduce from~\eqref{eq:Hopf-bis} that, for all $m$-box $B$, we also have as $n\to\infty$
\[
   \frac{n'_B}{N'}=\frac{\sigma(B)}{\sigma(\cd_{m})} + o(1).
\]
Let $B\in\B_m$. Applying the above equation for $SB$ and the left equality in~\eqref{eq:Hopf-bis} for $B$, and using~\eqref{eq:n_et_n'}, we get, if $\sigma(B)>0$,
\[
  \frac{N}{N'}=\frac{\sigma(SB)}{\sigma(B)}+o(1).
\]
It follows that the ratio $\sigma(SB)/\sigma(B)$ does not depend on $B$. We denote it by $c_m$.
Moreover, observe that if $\sigma(B)=0$, we get $n_B/N\to 0$, hence also $n_B/N'=n'_{SB}/N'\to 0$, and $\sigma(SB)=0$. Finally, for all $B\in\B_m$, we have $\sigma(SB)=c_m\sigma(B)$.

Note that any box $B\in\B_m$ is a finite disjoint union of $(m+1)$-boxes in $\B_{m+1}$. This implies that $c_m=c_{m+1}$. Therefore, there exists $c>0$ such that, for all $m$ large enough and all $B\in\B_m$, 
\[
  \sigma(SB)=c\sigma(B).
\]
But, as $m\to\infty$, the finite partition of $X^d$ defined by all $m$-boxes in $\B_m$ increases to the Borel $\sigma$-algebra of $X^d$. Hence, for any measurable subset $B\subset X^d$, the previous equality holds.

A direct application of Theorem~\ref{thm:product} proves that $\sigma$ has the product form announced in the statement of the proposition. And since $\sigma$ is boundedly finite, the measures $\sigma_{G_1}$ and $\sigma_{G_2}$ are also boundedly finite.
\end{proof}

\subsection{Proof of Theorem~\ref{thm:msj}, dissipative case}

We consider now a nonzero measure $\sigma$ on $X^d$, which is boundedly finite, $\td$-invariant, and such that the system $(X^d,\sigma,\td)$ is ergodic and totally dissipative. Up to a multiplicative constant, this measure is henceforth of the form 
\[  \sigma =  \sum_{k\in\ZZ} \delta_{(\td)^k x}     \]
for some $x\in X^d$. And since we assume that $\sigma$ is boundedly finite, for each $n$ there exist only finitely many $n$-crossings for $x$. Now we claim that for $n$ large enough, there is only one $n$-crossing for $x$, which will show that $\sigma$ is a diagonal measure.

Let $n$ be large enough so that $x\in\cd_{n-1}$, and let $m$ be large enough so that all $n$-crossings for $x$ are contained in a single $m$-crossing. Assume that there is a second $m$-crossings for $x$. Then we consider the smallest integer $\ell$ such that the first and the second $m$-crossings are contained in a single $(\ell+1)$-crossing. As in the proof of Lemma~\ref{lemma:special_n}, we have $t_\ell(x_i)\in\{1,2\}$, so we can apply Lemma~\ref{lemma:tn}. 
We get $(\td)^{h_\ell+1}x\in\cd_n$, but $h_\ell+1$ is necessarily in the second $m$-crossing. This contradicts the fact that all $n$-crossings for $x$ are contained in a single $m$-crossing. A similar argument proves that there is no other $m$-crossing contained in $\ZZ_-$, and this ends the proof of the theorem.

\section{Diagonal measures}
\label{sec:diagonal}

The purpose of this section is to provide more information on $d$-dimensional diagonal measures introduced in Definition~\ref{def:diagonal}, and which play an important role in our analysis. We are going to prove that there exist exactly two classes of ergodic diagonal measures:
\begin{itemize}
  \item graph joinings arising from powers of $T$, as defined by~\eqref{eq:graph_joining};
  \item \emph{weird} diagonal measures, whose marginals are singular with respect to $\mu$.
\end{itemize}
Moreover, we will provide a parametrization of the family of ergodic diagonal measures, and a simple criterion on the parameter to decide to which class a specific measure belongs.

\subsection{Construction of diagonal measures}
Let $d\ge2$, and let $\sigma$ be a diagonal measure on $X^d$. We define $n_0(\sigma)$ as the smallest integer $n_0$ for which $\sigma(C_{n_0-1}^d)>0$, and such that, for any $n\ge n_0$, $\sigma$ gives positive measure to a single $n$-diagonal, denoted by $D_n(\sigma)$. 

\begin{definition}
  \label{def:consistent}
  Let $n_0\ge 1$, and for each $n\ge n_0$, let $D_n$ be an $n$-diagonal. We say that the family $(D_n)_{n\ge n_0}$ is \emph{consistent} if 
  \begin{itemize}
    \item $\cd_{n_0-1}\cap\bigcap_{n\ge n_0}D_n\neq \emptyset$,
    \item $D_{n+1}\cap\cd_n\subset D_n$ for each $n\ge n_0$.
  \end{itemize}
\end{definition}
 
 Obviously, the family $(D_n(\sigma))_{n\ge n_0(\sigma)}$ is consistent. 
 
\begin{definition}
  \label{def:seen}
  We say that $x\in X^d$ is \emph{seen} by the consistent family of diagonals $(D_n)_{n\ge n_0}$ if, for each $n\ge n_0$, either $x\notin C_n^d$ (which happens only for finitely many integers $n$), or $x\in D_n$. We say that $x\in X^d$ is \emph{seen} by the diagonal measure $\sigma$ if it is seen by the family $(D_n(\sigma))_{n\ge n_0(\sigma)}$.
\end{definition}
 
 Observe that, thanks to the first condition in the definition of a consistent family of diagonals, there always exist some $x\in\cd_{n_0-1}$ which is seen by the family. Moreover, if $\sigma$ is a diagonal measure, then
\begin{equation}
  \label{eq:seen}\sigma\Bigl(\left\{x\in X^d:\ x\text{ is not seen by }\sigma\right\}\Bigr)=0.
\end{equation}

\begin{lemma}
  \label{lemma:seen}
  If $x$ is seen by the consistent family of diagonals $(D_n)_{n\ge n_0}$, then for each $j\in\ZZ$, 
  $(\td)^jx$ is also seen by $(D_n)_{n\ge n_0}$.
\end{lemma}
\begin{proof}
  Let $n\ge n_0$. Let $m\ge n$ be large enough so that $(\td)^ix$ belong to $\cd_m$ for each $0\le i\le j$ (or each $j\le i\le 0$). Consider the $m$-box $B$ containing $x$: Since $x$ is seen by $(D_n)_{n\ge n_0}$, $B\subset D_m$ and $(\td)^jB\subset D_m$.  Now, observe that an $m$-box is either contained in an $n$-box, or it is contained in $X^d\setminus \cd_n$. Hence, either $(\td)^j x\in (\td)^jB \subset\cd_n$, or $(\td)^j x\in (\td)^jB \subset   X^d\setminus \cd_n$. In the former case, $(\td)^j x\in (\td)^jB \subset D_n$ because $D_m\cap \cd_n\subset D_n$. This proves that $(\td)^jx$ is also seen by $(D_n)_{n\ge n_0}$.
\end{proof}

Let $(D_n)_{n\ge n_0}$ be a consistent family of diagonals. We want to describe the relationship between $D_n$ and $D_{n+1}$ for $n\ge n_0$. 

Let us consider an $n$-box $B$. For each $d$-tuple $\tau=(\tau(1),\ldots,\tau(d))\in\{1,2,3\}^d$,
\begin{equation}
  \label{eq:defB}B(\tau) \egdef \{x\in B:\ t_n(x_i)=\tau(i)\ \forall 1\le i \le d\}
\end{equation}
is an $(n+1)$-box. Moreover, $B$  is the disjoint union of the $3^d$ $(n+1)$-boxes $B(\tau)$.  Notice that if $B$ and $B'$ are two $n$-boxes included in the same $n$-diagonal, then $B(\tau)$ and $B'(\tau)$ are included in the same $(n+1)$-diagonal. Therefore, for each $n$-diagonal $D$ and each $d$-tuple $\tau\in\{1,2,3\}^d$, we can define the $(n+1)$-diagonal $D(\tau)$ as the unique $(n+1)$-diagonal containing $B(\tau)$ for any $n$-box $B$ included in $D$.

Let us fix $x\in \cd_{n_0-1}$ which is seen by $(D_n)_{n\ge n_0}$. For each $n\ge n_0$, since $x\in D_n\cap D_{n+1}$,  we get 
\[
  D_{n+1}=D_n(t_n(x_1),\ldots,t_n(x_d)).
\]

Moreover, we will see that some values for the $d$-tuple $(t_n(x_1),\ldots,t_n(x_d))$ are forbidden (see Figure~\ref{fig:diagonal}).  As a matter of fact, assume $\{1,2\}=\{t_n(x_i):\ 1\le i\le d\}$. We can apply Lemma~\ref{lemma:BtoSB}, and observe that the transformation $S$ used in this lemma acts as $T$ on some coordinates and as $\Id$ on others. Therefore, $x$ and $(\td)^{h_n+1}x$ belong to two different $n$-diagonals, which is impossible by Lemma~\ref{lemma:seen}. By a similar argument, we prove that the case $\{2,3\}=\{t_n(x_i):\ 1\le i\le d\}$ is also impossible. Eventually, only two cases can arise:
\begin{description}                                                                                                                                                                                                                                                                                                                                                                                                                                                                                                                                                                                             \item[Corner case] $\{1,3\}\subset \{t_n(x_i):\ 1\le i\le d\}$; then the first $(n+1)$-crossing for $x$ contains only one $n$-crossing for $x$.
\item[Central case] $t_n(x_1)=t_n(x_2)=\cdots = t_n(x_d)$; then the first $(n+1)$-crossing for $x$ contains three consecutive $n$-crossings for $x$, and $D_{n+1}=D_n(1,\ldots,1)=D_n(2,\ldots,2)=D_n(3,\ldots,3)$.
\end{description}

\begin{figure}[htp]
  \centering
  \includegraphics[width=10cm]{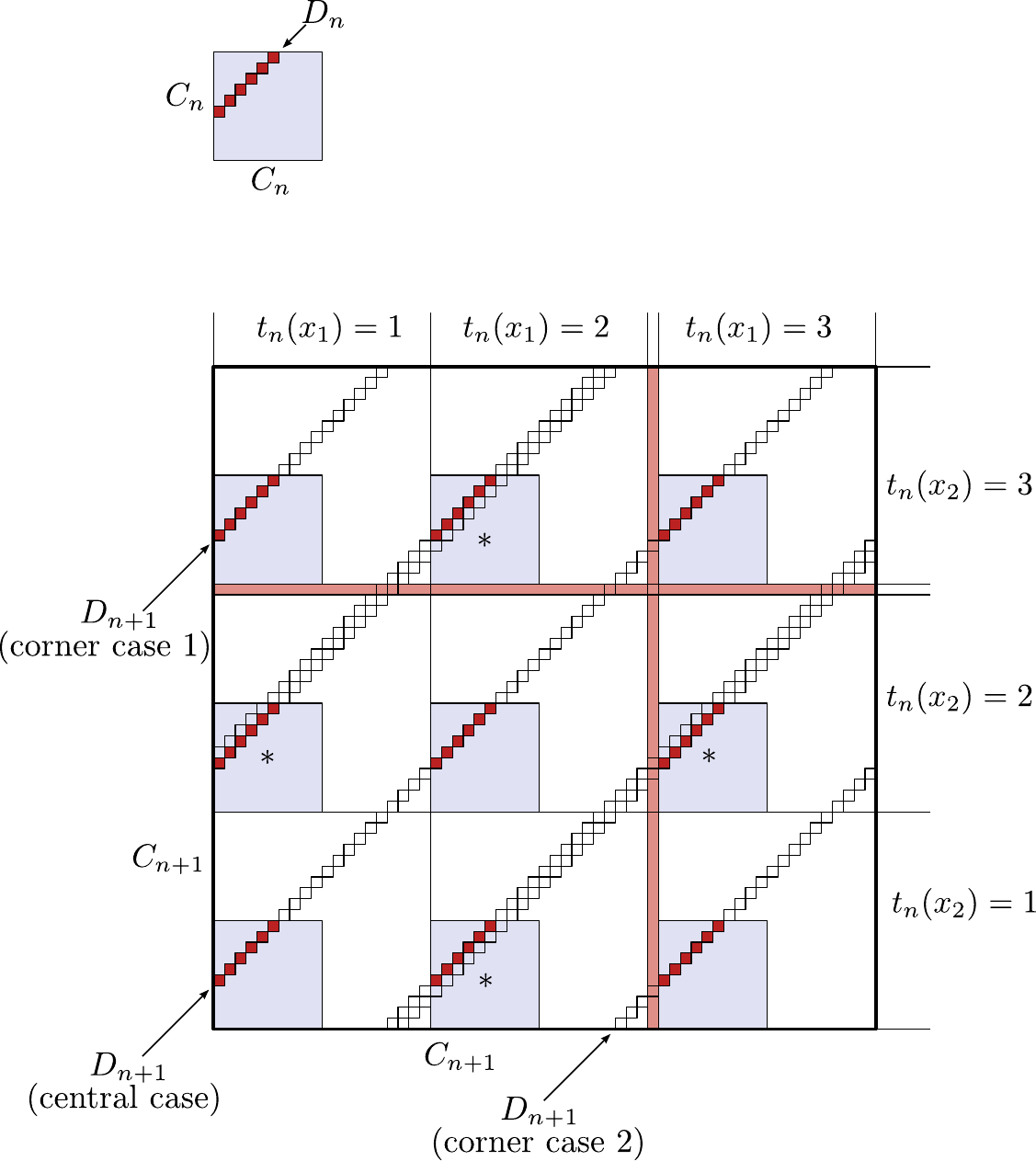}
  \caption{Relationship between $D_n$ and $D_{n+1}$ in the case $d=2$. The 4 positions marked with $\ast$ are impossible because the corresponding $(n+1)$-diagonal meets another $n$-diagonal.}
  \label{fig:diagonal}
\end{figure}

It follows from the above analysis that the diagonals $D_n$, $n\ge n_0$, are completely determined by the knowledge of $D_{n_0}$ and a family of parameters $(\tau_n)_{n\ge n_0}$, where each $\tau_n=(\tau_n(i), 1\le i\le d)$ is a $d$-tuple in $\{1,2,3\}^d$, satisfying either $\{1,3\}\subset \{\tau_n(i):\ 1\le i\le d\}$ (corner case), or $\tau_n(1)=\cdots = \tau_n(d)$ (central case).

\begin{lemma}
\label{lemma:central_case}
  If $\sigma$ is a diagonal measure, and if $(X^d,\td,\sigma)$ is conservative, then there are infinitely many integers $n$ such that the transition from $D_{n}(\sigma)$ to $D_{n+1}(\sigma)$ corresponds to the central case:
  \[
    D_{n+1}(\sigma)=D_{n}(\sigma)(1,\ldots,1).
  \]
\end{lemma}
\begin{proof}
 Since $(X^d,\td,\sigma)$ is conservative, for $\sigma$-almost all $x$, for any $n\ge n(x)$, there exist infinitely many $n$-crossings for $x$ in $\ZZ_+$. Moreover, $\sigma$-almost all $x$ is seen by $\sigma$. Applying Lemma~\ref{lemma:special_n} to such an $x$, we get that there are infinitely many integers $n$ for which the corner case does not occur, hence such that the transition from $D_{n}(\sigma)$ to $D_{n+1}(\sigma)$ corresponds to the central case.
\end{proof}

\begin{lemma}
  \label{lemma:intersection_of_n_boxes}
  Let $(\tau_m)_{m\ge m_0}$ be a sequence of $d$-tuples in $\{1,2,3\}^d$. We define a decreasing sequence of $m$-boxes by choosing an arbitrary $m_0$-box $B_{m_0}$ and setting inductively $B_{m+1}\egdef B_m(\tau_m)$. Then
  \[
    \bigcap_{m\ge m_0}B_m \neq \emptyset
  \]
  if and only if 
  \begin{equation}
    \label{eq:condition_tau}
    \text{for all }1\le i\le d, \text{ there exist infinitely many integers }m\text{ with } \tau_m(i)\in\{1,2\}.
  \end{equation}
\end{lemma}
\begin{proof}
  Recall that the levels of each tower in the construction of $T$ are intervals which are closed to the left and open to the right. If we have a decreasing sequence $(I_m)$ of intervals, where $I_m$ is a level of tower~$m$, then 
  \[ \bigcap_m I_m=\begin{cases}
                     \emptyset, \text{ if $I_{m+1}$ is the rightmost subinterval of $I_m$ for each large enough $m$,}\\
                     \text{a singleton, otherwise.}
                   \end{cases}
                   \]
   Since $\tau_m(i)$ indicates the subinterval chosen at step $m$ for the coordinate $i$, the conclusion follows. 
\end{proof}

\begin{lemma}
\label{lemma:consistent}
  Let $n_0\ge 2$. Let $D_{n_0}$ be an $n_0$-diagonal such that $D_{n_0}\cap \cd_{n_0-1}\neq \emptyset$.
  Let $(\tau_n)_{n\ge n_0}$ be a sequence of $d$-tuples  in $\{1,2,3\}^d$ satisfying either $\{1,3\}\subset \{\tau_n(i):\ 1\le i\le d\}$, or $\tau_n(1)=\cdots = \tau_n(d)$.
  Then the inductive relation $D_{n+1}\egdef D_n(\tau_n)$, $n\ge n_0$ defines a consistent family of diagonals if and only if Property~\eqref{eq:condition_tau} holds. 
\end{lemma}
\begin{proof}
  Applying Lemma~\ref{lemma:intersection_of_n_boxes}, the first condition in the definition of a consistent family of diagonals is equivalent to Property~\eqref{eq:condition_tau}. The second condition comes from the restrictions made on the $d$-tuples.
\end{proof}

\begin{prop}
  \label{prop:construction_sigma}
  Let $n_0\ge2$. Let $(D_n)_{n\ge n_0}$ be a consistent family of diagonals.
  Then there exists a diagonal measure $\sigma$, unique up to a multiplicative constant, with $n_0(\sigma)\le n_0$, and for each $n\ge n_0$, $D_n(\sigma)=D_n$. 
  This measure satisfies $\sigma(X^d)=\infty$.
  
  If the transition from $D_n$ to $D_{n+1}$ corresponds infinitely often to the central case, then the system $(X^d, \td, \sigma)$ is conservative ergodic. Otherwise, it is ergodic and totally dissipative.
\end{prop}

\begin{proof}
  We first define $\sigma$ on the ring
  \[
    \mathscr{R}\egdef\{B\subset X^d:\ \exists n\ge1,\ B\text{ is a finite union of $n$-boxes}\}.
  \]
  Since we want to determine $\sigma$ up to a multiplicative constant, we can arbitrarily set $\sigma(\cd_{n_0})=\sigma(D_{n_0})\egdef 1$. As we want $\sigma$ to be invariant under the action of $\td$, this fixes the measure of each $n_0$-box: For each $n_0$-box $B$, 
  \[
    \sigma(B)\egdef\begin{cases}
                     \dfrac{1}{\text{number of $n_0$-boxes in $D_{n_0}$}} &\text{ if }B\subset D_{n_0},\\
                     0 &\text{otherwise.}
                   \end{cases}
  \]
  Now assume that we have already defined $\sigma(B)$ for each $n$-box, for some $n\ge n_0$, and that we have some constant $\alpha_n>0$ such that, for any $n$-box $B$,
  \[
    \sigma(B)=\begin{cases}
                     \alpha_n &\text{ if }B\subset D_{n},\\
                     0 &\text{otherwise.}
                   \end{cases}
  \]
We set $\sigma(B')\egdef 0$ for any $(n+1)$-box $B'\not\subset D_{n+1}$, and it remains to define the measure of $(n+1)$-boxes included in $D_{n+1}$. These boxes must have the same measure, which we denote by $\alpha_{n+1}$. 
  \begin{itemize}
    \item Either the transition from $D_n$ to $D_{n+1}$ corresponds to the corner case. Then each $n$-box contained in $D_n$ meets only one $(n+1)$-box contained in $D_{n+1}$, and we set $\alpha_{n+1}\egdef\alpha_n$.
    \item Or  the transition from $D_n$ to $D_{n+1}$ corresponds to the central case. Then each $n$-box contained in $D_n$ meets three  $(n+1)$-boxes contained in $D_{n+1}$,  and we set $\alpha_{n+1}\egdef\alpha_n/3$.
  \end{itemize}
  For any $R\in\mathscr{R}$ which is a finite union of $n$-boxes, we can now define $\sigma(R)$ as the sum of the measures of the $n$-boxes included in $R$. At this point, $\sigma$ is now defined as a finitely additive set function on $\mathscr{R}$.
  
  It remains now to prove that $\sigma$ can be extended to a measure on the Borel $\sigma$-algebra of $X^d$, which is the $\sigma$-algebra generated by $\mathscr{R}$. Using Theorems~F p.~39 and A p.~54 (Caratheodory's extension theorem) in~\cite{Halmos}, we only have to prove the following.
  \begin{claim*}
   If $(R_k)_{k\ge 1}$ is a decreasing sequence in $\mathscr{R}$ such that 
  $  \lim_{k\to\infty}\downarrow \sigma(R_k) > 0$,
  then $\bigcap_k R_k\neq \emptyset$.
 \end{claim*} 
Having fixed such a sequence $(R_k)$, we say that an $m$-box $B$ is \emph{persistent} if 
\[
  \lim_{k\to\infty}\downarrow \sigma(R_k\cap B) > 0.
\]
We are going to construct inductively a decreasing family $(B_m)_{m\ge m_0}$ where $B_m$ is a persistent $m$-box and 
\[
\emptyset\neq \bigcap_{m\ge m_0}B_m\subset \bigcap_{k} R_k.
\]
We first consider the case where the transition from $D_n$ to $D_{n+1}$ corresponds infinitely often to the central case. Choose $k_0$ large enough so that 
\[
  \sigma(R_{k_0}) <  \frac{3}{2}\lim_{k\to\infty}\downarrow \sigma(R_k).
\]
Then there exists $m_0$ such that $R_{k_0}$ is a finite union of $m_0$-boxes, and (choosing a larger $m_0$ if necessary), the transition from $D_{m_0}$ to $D_{m_0+1}$ corresponds to the central case. Let $B$ be a persistent $m_0$-box. Then $\sigma$ on $B$ is concentrated on the $(m_0+1)$-boxes $B(1,\ldots,1)$, $B(2,\ldots,2)$ and $B(3,\ldots,3)$. If $B(1,\ldots,1)$ is not persistent, we get
\begin{align*}
  0 < \lim_{k\to\infty}\downarrow \sigma(R_k \cap B) & = \lim_{k\to\infty}\downarrow \sigma(R_k \cap B(2,\ldots,2)) +  \lim_{k\to\infty}\downarrow \sigma(R_k \cap B(3,\ldots,3))\\
  & \le \sigma(R_{k_0} \cap B(2,\ldots,2)) + \sigma(R_{k_0} \cap B(3,\ldots,3)) \\
  & \le \sigma(B(2,\ldots,2)) + \sigma(B(3,\ldots,3)) \\
  & = \frac{2}{3} \sigma (B) = \frac{2}{3} \sigma (R_{k_0}\cap B).
\end{align*}
Therefore, there exists some persistent $m_0$-box $B_{m_0}$ such that $B_{m_0+1}\egdef B_{m_0}(1,\ldots,1)$ is also persistent. 
Indeed, otherwise we would have
\begin{align*}
  \sigma(R_{k_0}) &\ge \sum_{B \text{persistent $m_0$}-box} \sigma(R_{k_0}\cap B)\\
  & \ge \frac{3}{2}  \sum_{B \text{persistent $m_0$}-box}\lim_{k\to\infty}\downarrow \sigma(R_k\cap B)\\
  & = \frac{3}{2}\lim_{k\to\infty}\downarrow \sigma(R_k),
\end{align*}
which would contradict the definition of $k_0$. 

Assume that we have already defined $B_{m_i}$ and $B_{m_i+1}=B_{m_i}(1,\ldots,1)$ for some $i\ge 0$.
Then we choose $k_{i+1}$ large enough so that 
\[
  \sigma(R_{k_{i+1}}\cap B_{m_i+1}) <  \frac{3}{2}\lim_{k\to\infty}\downarrow \sigma(R_k\cap B_{m_i+1}).
\]
We choose $m_{i+1}>m_i+1$ such that $R_{k_{i+1}}$ is a finite union of $m_{i+1}$-boxes, and the transition from $D_{m_{i+1}}$ to $D_{m_{i+1}+1}$ corresponds to the central case. Then the same argument as above, replacing $R_k$ by $R_k\cap  B_{m_i+1}$, proves that there exists a persistent $m_{i+1}$-box $B_{m_{i+1}}\subset B_{m_i+1}$ such that $B_{m_{i+1}+1}\egdef B_{m_{i+1}}(1,\ldots,1)$ is also persistent. 

Now we can complete in a unique way our sequence to get a decreasing sequence $(B_m)_{m\ge m_0}$ of persistent boxes. Since we have $B_{m_i+1}=B_{m_i}(1,\ldots,1)$ for each $i\ge0$, Lemma~\ref{lemma:intersection_of_n_boxes} ensures that 
\[
  \bigcap_m B_m\neq\emptyset.
\]
It only remains to prove that $\bigcap_m B_m \subset\bigcap_k R_k$. Indeed, let us fix $k$ and let $\overline{m}$ be such that $R_k$ is a finite union of $\overline{m}$-boxes. In particular, $R_k$ contains all persistent $\overline{m}$-boxes, which implies
\[
  \bigcap_m B_m \subset B_{\overline{m}} \subset R_k.
\]

Now we consider the case where there exists $m_0\ge n_0$ such that, for $n\ge m_0$, the transition from $D_n$ to $D_{n+1}$ always correspond to the corner case. That is, there exists a family $(\tau_n)_{n\ge m_0}$ of $d$-tuples in $\{1,2,3\}$, with $\{1,3\}\subset\{\tau_n(i), 1\le i\le d\}$ for each $n\ge m_0$, such that $D_{n+1}=D_n(\tau_n)$. By Lemma~\ref{lemma:consistent}, property~\eqref{eq:condition_tau} holds for $(\tau_n)_{n\ge m_0}$. We will now construct the family $(B_m)_{m\ge m_0}$ of $m$-boxes satisfying the required conditions. Start with $B_{m_0}$ which is a persistent $m_0$-box (such a box always exists). Since the transition from $D_{m_0}$ to $D_{m_0+1}$ corresponds to the corner case, there is only one $(m_0+1)$-box contained in $D_{m_0+1}\cap B_{m_0}$, and this box is precisely $B_{m_0}(\tau_{m_0})$. Therefore this box is itself persistent, and defining inductively $B_{m+1}\egdef B_m(\tau_m)$ gives a decreasing family of persistent boxes. By Lemma~\ref{lemma:intersection_of_n_boxes}, $\
bigcap_{m\ge m_0} B_m\neq\emptyset$. We prove as in the preceding case that $\bigcap_m B_m \subset\bigcap_k R_k$. This ends the proof of the claim.

\medskip

This proves that $\sigma$ can be extended to a $\td$-invariant measure, whose restriction to each $\cd_n$, $n\ge n_0$, is by construction concentrated on the single diagonal $D_n$. And since $\cd_{n_0-1}\cap D_{n_0}\neq\emptyset$, we get $n_0(\sigma)\le n_0$. If $B$ is an $n$-box, then $(\td)^{h_n/2}B\subset\cd_{n+1}$. Moreover, by Lemma~\ref{lemma:separated}, $(\td)^{h_n/2}B\not\subset\cd_{n}$. 
It follows that  $(\td)^{h_n/2}D_n\subset\cd_{n+1}\setminus\cd_n$. But $\sigma\bigl((\td)^{h_n/2}D_n\bigr)=\sigma(D_n)$, hence $\sigma(\cd_{n+1})\ge 2\sigma(\cd_n)$. We conclude that $\sigma(X^d)=\infty$.

Now we want to show the ergodicity of the system 
$(X^d,\td,\sigma)$. Let $A\subset X^d$ be a $\td$-invariant measurable set, with $\sigma(A)\neq 0$. Let $n$ be such that $\sigma(A\cap \cd_n)>0$.  Given $\varepsilon>0$, we can find $m>n$ large enough such that there exists $\tilde A$, a finite union of $m$-boxes, with
\[
  \sigma \left( (A\vartriangle \tilde A)\cap\cd_n\right) < \varepsilon\,\sigma(A\cap \cd_n).
\]
Let $B$ be an $m$-box in $D_m$, and set $s_m\egdef \sigma(A\cap B)$: By invariance of $A$ under the action of $\td$, $s_m$ does not depend on the choice of $B$.
We have
\[
  \sigma(A\cap \cd_n) = \sum_{\stack{B\ m\text{-box in }D_m}{B\subset\cd_n}}\sigma(A\cap B) = s_m \cdot \left|
  \{B\ m\text{-box}:\ B\subset D_m\cap \cd_n\}\right|.
\]
On the other hand, we can write
\[
  s_m \cdot \left|
  \{B\ m\text{-box}:\ B\subset D_m\cap \cd_n\setminus\tilde A\}\right| \le \sigma \left( (A\vartriangle \tilde A)\cap\cd_n\right) < \varepsilon\,\sigma(A\cap \cd_n).
\]
It follows that
\[
  \dfrac{\Bigl|\{B\ m\text{-box}:\ B\subset D_m\cap \cd_n\setminus\tilde A\}\Bigr|}{\Bigl|
  \{B\ m\text{-box}:\ B\subset D_m\cap \cd_n\}\Bigr|} < \varepsilon,
\]
hence
\[
  \sigma(\tilde A\cap\cd_n) > (1-\varepsilon) \sigma(\cd_n),
\]
and finally
\[
  \sigma(A\cap\cd_n) > (1-2\varepsilon) \sigma(\cd_n).
\]
But this holds for any $\varepsilon>0$, which proves that $\sigma(A\cap\cd_n) = \sigma(\cd_n)$. Again, this holds for any large enough $n$, thus $\sigma(X^d\setminus A)=0$, and the system is ergodic.

We can observe that, if the central case occurs infinitely often, the measure $\alpha_n$ of each $n$-box on $D_n$ decreases to 0 as $n$ goes to infinity, which ensures that $\sigma$ is continuous. Therefore the conservativity of $(X^d,\td,\sigma)$ is a consequence of the ergodicity of this system. On the other hand, if the central case occurs only finitely many times, there exists $m_0$ such that for each $m\ge m_0$, $\alpha_m=\alpha_{m_0}>0$. It follows that $\sigma$ is purely atomic,  and by ergodicity of $(X^d,\td,\sigma)$, $\sigma$ is concentrated on a single orbit. 
\end{proof}

 \subsection{A parametrization of the family of  diagonal measures}
 \label{sec:parametrization}

  If $\sigma$ is a diagonal measure, by definition of $n_0(\sigma)$, the diagonal $D_{n_0(\sigma)}(\sigma)$ is initial in the sense given by the following definition.
  \begin{definition}
  Let $n_0\ge1$, and $D$ an $n_0$-diagonal. We say that $D$ is an \emph{initial} diagonal if
  \begin{itemize}
    \item Either there exist at least two $(n_0-1)$-diagonals which have non-empty intersection with $D$;
    \item Or $D$ has non-empty intersection with exactly one $(n_0-1)$-diagonal, but does not intersect $\cd_{n_0-2}$ (with the convention that $\cd_0=\emptyset$).
  \end{itemize}
 \end{definition}

 In Proposition~\ref{prop:construction_sigma}, it is clear that $n_0=n_0(\sigma)$ if and only if $D_{n_0}$ is initial.
 
 Now we are able to provide a canonical parametrization of the family of diagonal measures: we consider the set of parameters
 \[
   \D \egdef\Bigl\{(n_0, D,\tau)\Bigr\},
 \]
 where
 \begin{itemize}
    \item $n_0\ge 1$,
    \item $D$ is an initial $n_0$-diagonal;
    \item $\tau= (\tau_n)_{n\ge n_0}$, where for each $n\ge n_0$, $\tau_n\in\{1,2,3\}^d$ and satisfies either $\{1,3\}\subset\{\tau_n(i),\ 1\le i\le d\}$ (corner case), or $\tau_n(i)=1$ for each $1\le i\le d$ (central case);
    \item Property~\eqref{eq:condition_tau} holds for $(\tau_n)$.
 \end{itemize}
To each $(n_0,D,\tau)\in\D $, by Proposition~\ref{prop:construction_sigma} we can canonically associate an ergodic diagonal measure $\sigma_{(n_0,D,\tau)}$, setting $\sigma_{(n_0,D,\tau)}(\cd_{n_0})\egdef 1$. Conversely, any ergodic diagonal measure $\sigma$ can be written as
\[
  \sigma = \lambda\,\sigma_{(n_0,D,\tau)} 
\]
for some $(n_0,D,\tau)\in\D $, where $\lambda\egdef \sigma(\cd_{n_0(\sigma)})$, $n_0\egdef n_0(\sigma)$, and $D\egdef D_{n_0(\sigma)}(\sigma)$.

Note that, by construction, for each $n\ge1$, each $(n_0,D,\tau)\in\D$, and each $n$-box $B$, we have $\sigma_{(n_0,D,\tau)}(B)\le1$. Thus,
\begin{equation}
  \label{eq:borne_universelle}
  \forall n\ge1,\ \forall (n_0,D,\tau)\in\D,\ \sigma_{(n_0,D,\tau)}(\cd_n)\le \left(\dfrac{h_n}{2}\right)^d. 
\end{equation}

\subsection{Identification of graph joinings}

\begin{prop}
\label{prop:graph}
  Graph joinings of the form
\begin{equation}
  \label{eq:graph_joining}\sigma(A_1\times\cdots\times A_d)  = \alpha\, \mu (A_1\cap T^{-k_2}(A_2)\cap\cdots\cap T^{-k_d}(A_d))
\end{equation}
for some real $\alpha>0$ and some integers $k_2,\ldots,k_d$, are the diagonal measures $\sigma_{(n_0,D,\tau)}$ for which there exists $n_1\ge n_0$ such that, for $n\ge n_1$, $\tau_n(i)=1$ for each $1\le i\le d$.
\end{prop}

\begin{proof}
 Let $\sigma\egdef\sigma_{(n_0,D,\tau)}$, and assume that for $n\ge n_1$, $\tau_n(i)=1$ for each $1\le i\le d$. 
 Consider $n\ge n_1$, and let $B$ be an $n$-box in $D_{n}(\sigma)$. Then $B$ is of the form $B_1\times T^{k_2}B_1\times\cdots\times T^{k_d}B_1$ for some level $B_1$ of tower~$n$ and some integers $k_2,\ldots,k_d$. Moreover, $k_2,\ldots,k_d$ do not depend on the choice of $B$ in $D_{n}(\sigma)$. Let us also write $B$ as $T^{\ell_1}F_{n}\times \cdots\times T^{\ell_d}F_{n}$, where $F_n$ is the bottom level of tower~$n$. Then $k_i=\ell_i-\ell_1$ for each $2\le i\le d$. Now, recalling notation~\eqref{eq:defB}, consider $B(1,\ldots,1)$, which is an $(n+1)$-box in $D_{n+1}(\sigma)$. Then  $B(1,\ldots,1)=T^{\ell_1}F_{n+1}\times \cdots\times T^{\ell_d}F_{n+1}$, thus this $(n+1)$-box is of the form $B'_1\times T^{k_2}B'_1\times\cdots\times T^{k_d}B'_1$, for some level $B'_1$ in tower~$(n+1)$, and the \emph{same} integers $k_2,\ldots,k_d$ as above. By induction, this is true for any $n$-box in $D_n(\sigma)$ for any $n\ge n_1$. 
 As in the proof of Proposition~\ref{prop:construction_sigma}, let us denote by $\alpha_n$ the measure of each $n$-box in $D_{n}(\sigma)$. By hypothesis, all transitions from $n_1$ correspond to the central case hence for each $n\ge n_1$, $\alpha_n=\alpha_{n_1}/3^{n-n_1}$. 
 
 Fix $n\ge n_1$ and consider some $n$-box $B$, of the form $B=A_1\times A_2\times\cdots\times A_d$  for sets $A_i$ which are levels of tower~$n$. We have
\[
  \sigma(B) = \begin{cases}
                \alpha_{n_1}/3^{n-n_1} &\text{ if } A_1\cap T^{-k_2}(A_2)\cap\cdots\cap T^{-k_d}(A_d)=A_1\\
                0 &\text{ otherwise, that is if } A_1\cap T^{-k_2}(A_2)\cap\cdots\cap T^{-k_d}(A_d)=\emptyset.
              \end{cases}
\]
Observing that $\mu(A_1)=\mu(F_{n_1})/3^{n-n_1}$, we get
\[
  \sigma(A_1\times A_2\times\cdots\times A_d) = \alpha \mu (A_1\cap T^{-k_2}(A_2)\cap\cdots\cap T^{-k_d}(A_d)),
\]
with $\alpha\egdef \alpha_{n_1}/\mu(F_{n_1})$. Finally, the above formula remains valid if the sets $A_i$ are finite unions of levels of tower~$n$, then for any choice of these sets.

Conversely, assume that $\sigma$ is a graph joining of the form given by~\eqref{eq:graph_joining}.  Observe that if $A$ is a level of $C_n$, and if $|k|\le h_n$, then 
\[A\cap T^k A=
  \begin{cases}
    A &\text{ if } k=0,\\
    \emptyset&\text{ otherwise.}
  \end{cases}
\]
Take $n$ large enough so that for all $1\le i\le d$, $h_n/2 > |k_i|$. Let $B$ be an $n$-box, which can always be written as $B=A\times T^{k'_2}A_2\times\cdots\times T^{k'_d}A_d$ for some level $A$ of $C_n$ and some integers $k'_2,\ldots,k'_d$ satisfying $|k'_i|\le h_n/2$. Then 
\[ 
\sigma(B) = \alpha \mu(A\cap T^{k'2-k_2}(A)\cap\cdots\cap T^{k'd-k_d}(A)),
\]
which is positive if and only if for each $1\le i\le d$, $k_i=k'_i$. Hence $\sigma|_{\cd_n}$ is concentrated on a single diagonal, which is constituted by $n$-boxes of the form $A\times T^{k_2}(A)\times\cdots\times T^{k_d}(A)$. This already proves that $\sigma$ is a diagonal measure.  Moreover, if $B$ is such an $n$-box, then $B(1,\ldots,1)$ is an $(n+1)$-box of the same form, hence the transition from $n$ to $n+1$ corresponds to the central case.
\end{proof}

\begin{definition}
  We say that $x_1\in X$ is \emph{compatible} with the diagonal measure $\sigma_{(n_0,D,\tau)}$ if there exists $(x_2,\ldots,x_d)\in X^{d-1}$ such that $(x_1,\ldots,x_d)$ is seen by $\sigma_{(n_0,D,\tau)}$.
\end{definition}

\begin{prop}
\label{prop:compatible}
  Let $\sigma_{(n_0,D,\tau)}$ be a diagonal measure. If the set of $x_1\in X$  which are compatible with  $\sigma_{(n_0,D,\tau)}$ is of positive measure $\mu$, then $\sigma_{(n_0,D,\tau)}$ is a graph joining arising from powers of $T$, as defined by~\eqref{eq:graph_joining}.
\end{prop}

\begin{proof}
  Let $x_1$ be compatible with the diagonal measure $\sigma\egdef\sigma_{(n_0,D,\tau)}$, and let $(x_2,\ldots,x_d)\in X^{d-1}$ be such that $(x_1,\ldots,x_d)$ is seen by $\sigma$. Let $n\ge n_0$ be large enough so that $(x_1,\ldots,x_d)\in \cd_n$. Then 
  \[(x_1,\ldots,x_d)\in D_{n+1}(\sigma)=D_n(\sigma)(\tau_n(1),\ldots,\tau_n(d)).
      \]
 If we further assume that $(\tau_n(1),\ldots,\tau_n(d))\neq(1,\ldots,1)$, then the transition from $D_n(\sigma)$ to $D_{n+1}(\sigma)$ corresponds to the corner case, and there is only one occurrence of $D_n(\sigma)$ inside $D_{n+1}(\sigma)$. Since also $(x_1,\ldots,x_d)\in D_{n}(\sigma)$, it follows that $t_n(x_1)=\tau_n(1)$. 
  Therefore, if there exist infinitely many integers $n$ such that $(\tau_n(1),\ldots,\tau_n(d))\neq(1,\ldots,1)$, then the compatibility of $x_1$ with the diagonal measure $\sigma$ forces the value of $t_n(x_1)$ for infinitely many integers $n$. This implies that $x_1$ belongs to a fixed set which is $\mu$-negligible. 
  
  To conclude the proof, it is enough to apply Proposition~\ref{prop:graph}.
\end{proof}

\begin{remark}
\label{remark:weird}
  Taking $(n_0,D,\tau)\in\D$ for which the corner case occurs infinitely often, and considering the corresponding diagonal measure $\sigma_{(n_0,D,\tau)}$, we see that there exist ergodic diagonal measures which are not graph joinings. By Proposition~\ref{prop:compatible}, these measures are concentrated on sets $N_1\times N_2\times \cdots\times N_d$, where each $N_i$, $1\le i\le d$, is a $\mu$-negligible set. We call such a measure a \emph{weird} measure. 
  It is conservative whenever the central case occurs infinitely often. 
\end{remark}

\section{Joinings and consequences}
\label{sec:joinings}
Although weird measures have marginals which are singular with respect to $\mu$, Danilenko has shown in~\cite[Example~5.4]{RadonMSJ} that we can get some conservative $T\times T$-invariant measure on $X\times X$ with absolutely continuous marginals by taking an appropriate convex combination of weird measures.
However, if we restrict ourselves to true joinings (see the definition below), then this phenomenon does not occur, and only graph measures can appear as ergodic components.

Let $(Y_i,\B_i,\nu_i,S_i)$, $i=1,2$, be two infinite measure preserving dynamical systems.
We recall that a \emph{joining} between them is any $S_1\times S_2$-invariant measure $m$ on the Cartesian product $Y_1\times Y_2$, whose marginals are respectively $\nu_1$ and $\nu_2$.

\begin{prop}
  \label{prop:alphamu}
If $m$ is a joining between $\left(X,\A,\mu,T\right)$ and $\left(X,\A,\alpha\mu,T\right)$ for some $\alpha >0$ then $\alpha=1$ and $m$ is a convex combination of graph measures supported by powers of $T$.
\end{prop}
\begin{proof}
 Assume $m$ is such a joining between $\left(X,\A,\mu,T\right)$ and $\left(X,\A,\alpha\mu,T\right)$. Then $m$ is $T\times T$-invariant, and its marginals are $\mu$ and $\alpha\mu$ respectively. Observe this is also true for any of its ergodic components. Since no weird measures (thanks to Remark~\ref{remark:weird}) nor the product measure have such marginals, only graph measures supported by powers of $T$ can appear in the ergodic decomposition.
Therefore, there exist nonnegative numbers $a_{k}\in\ZZ$  such that the ergodic decomposition of $m$ writes
\[
m\left(A_{1}\times A_{2}\right)=\sum_{k\in\mathbb{Z}}a_{k} \, \mu\left(A_{1}\cap T^{-k}A_{2}\right).
\]
Considering the first marginal of $m$ which is $\mu$, we get $\sum_{k\in\mathbb{Z}}a_{k}=1$, and the second marginal gives $\sum_{k\in\mathbb{Z}}a_{k}=\alpha$
, thus $\alpha=1$.
\end{proof}

As an immediate consequence, we obtain:

\begin{prop}
  The centralizer of $T$ is reduced to the powers of $T$.
\end{prop}

Proposition~\ref{prop:alphamu} also leads to a nice corollary, for which we need to recall from~\cite{Aaronson} the following definition.

\begin{definition}
  A \emph{law of large numbers} for a conservative, ergodic, measure preserving dynamical system $(Y,\B,\nu,S)$ is a function $L:\{0,1\}^\NN\to[0,\infty]$ such that
  for all $B\in\B$, for $\nu$-almost every $y\in Y$,
  \[ 
    L\left(\ind{B}(y),\ind{B}(Sy),\ldots\right) = \nu(B).
  \]
\end{definition}

Theorem 3.2.5 in~\cite{Aaronson} provides a sufficient condition for $S$ to have a law of large numbers, which is exactly the conclusion of Proposition~\ref{prop:alphamu}.

\begin{corollary}
\label{cor:law_of_large_numbers}
  The dynamical system $\left(X,\A,\mu,T\right)$ has a law of large numbers.
\end{corollary}

\begin{prop}
\label{prop:joining with Chacon}
Let $\left(Z,\mathcal{Z},\rho,R\right)$ be any dynamical system, and assume that there exists a joining $\left(X\times Z,\mathcal{A}\otimes\mathcal{Z},m,T\times R\right)$.
Then $\left(X,\A,\mu,T\right)$ is a factor of $\left(Z,\mathcal{Z},\rho,R\right)$.
\end{prop}
\begin{proof}
Since the marginal of $m$ on the second coordinate is $\rho$, there exists a family $(\mu_z)_{z\in Z}$ of probability measures on $X$ (defined $\rho$-almost everywhere), such
that we have the following disintegration of $m$: for all $A\in\A$ and all $B\in\mathcal{Z}$,
\[
m(A\times B) = \int_B \mu_z(A)\, d\rho(z).
\]
Since $m$ is $T\times R$-invariant, we have $\rho$-almost everywhere
\begin{equation}
  \label{eq:mu_Rz}
  \mu_{Rz}=T_*(\mu_z).  
\end{equation}
We can then form the relatively independent joining of $\left(X\times Z,\mathcal{A}\otimes\mathcal{Z},m,T\times R\right)$
over $\left(Z,\mathcal{Z},\rho,R\right)$, that is:
\[
\left(X\times Z\times X,\mathcal{A}\otimes\mathcal{Z}\otimes\mathcal{A},m\otimes_{\mathcal{Z}}m,T\times R\times T\right),
\]
where
\[
m\otimes_{\mathcal{Z}}m\left(A_{1}\times B\times A_{2}\right)=\int_{B}\mu_{z}\otimes\mu_{z}\left(A_{1}\times A_{2}\right)\rho\left(dz\right),
\]
and extract from it a self-joining $\left(X\times X,\mathcal{A}\otimes\mathcal{A},\widetilde{m},T\times T\right)$
where
\[
\widetilde{m}\left(A_{1}\times A_{2}\right)=\int_{Z}\mu_{z}\otimes\mu_{z}\left(A_{1}\times A_{2}\right)\rho\left(dz\right).
\]
Then $\widetilde{m}$ is $T\times T$-invariant, and its marginals are both equal to $\mu$. 
As in the proof of Proposition~\ref{prop:alphamu}, we deduce that there exist nonnegative numbers $a_{k}\in\ZZ$ with $\sum_{k\in\mathbb{Z}}a_{k}=1$ such that the ergodic decomposition of $\widetilde{m}$ writes
\[
\int_{Z}\mu_{z}\otimes\mu_{z}\left(A_{1}\times A_{2}\right)\rho\left(dz\right)=\sum_{k\in\mathbb{Z}}a_{k} \, \mu\left(A_{1}\cap T^{-k}A_{2}\right).
\]
For $\rho$-a.e. $z\in Z$, the probability measure $\mu_{z}\otimes\mu_{z}$ is therefore supported
by the graphs of $T^{k}$, $k\in\mathbb{Z}$. In particular, $\mu_{z}$
is a discrete probability measure, and its support is necessarily contained in
a single $T$-orbit. This support can be totally ordered according
to the place on the orbit, thus
we can measurably choose one point $\varphi(z)$ on the support of $\mu_z$ by looking at the
point with the highest weight and the lowest place in the orbit (this
is well defined as the number of such points is finite). Since $\mu_z$ is supported by the $T$-orbit of $\varphi(z)$,
we have a family $(w_i)_{i\in\ZZ}$ of measurable functions from $Z$ to $[0,1]$ such that, for $\rho$-almost every $z$, 
\[
  \mu_z = \sum_{i\in\ZZ} w_i(z) \, \delta_{T^i\varphi(z)}.
\]
Then, the disintegration of $m$ becomes
\begin{equation}
  \label{eq:disintegration2}
  m(A\times B) = \sum_{i\in\ZZ} \int_B w_i(z) \ind{A}\bigl(T^i\varphi(z)\bigr)\,d\rho(z).
\end{equation}
Of course, since $\mu_z$ is a probability, we have $\sum_{i\in\ZZ} w_i(z) =1$, $\rho$-almost everywhere.
Moreover, from~\eqref{eq:mu_Rz}, we deduce that $\varphi\circ R=T\circ \varphi$, and that each function $w_i$ is $R$-invariant.
To show that $\varphi$ is a homomorphism between the dynamical systems $\left(Z,\mathcal{Z},\rho,R\right)$ and $\left(X,\A,\mu,T\right)$, it only remains to check
that $\varphi_*(\rho)=\mu$. But this comes from the following computation: for each $A\in\A$, we have
\begin{align*}
  \rho\bigl(\varphi^{-1}(A)\bigr) & = \int_Z \ind{A}\bigl(\varphi(z)\bigr)\,d\rho(z) \\
  & =  \int_Z  \sum_{i\in\ZZ} w_i(z) \ind{A}\bigl(\varphi(z)\bigr)\,d\rho(z)\\
  & =  \sum_{i\in\ZZ} \int_Z  w_i(R^i z) \ind{A}\bigl(\varphi(R^i z)\bigr)\,d\rho(z) \quad\text{(by $R$-invariance of $\rho$)}\\
  & =  \sum_{i\in\ZZ} \int_Z  w_i(z) \ind{A}\bigl(T^i\varphi(z)\bigr)\,d\rho(z) \\
  & = m(A\times Z)  \quad\text{(by~\eqref{eq:disintegration2})}\\
  & = \mu(A).
\end{align*}
\end{proof}

\begin{prop}[$T$ has no non-trivial factor]
  \label{prop:sans_facteur}
  Assume that $\left(Z,\mathcal{Z},\rho,R\right)$ is a  factor of $\left(X,\A,\mu,T\right)$. 
  Then any homomorphism $\pi:X\to Z$ between the two systems is in fact an isomorphism.
\end{prop}

\begin{proof}
  To any homomorphism $\pi:X\to Z$, we can associate the joining $\Delta_\pi$ of the two systems defined by
  \[
    \Delta_\pi(A\times B)\egdef \mu(A\cap \pi^{-1}B)
  \]
  for any $A\in\A$, $B\in\mathcal{Z}$. Let us repeat the construction made in the proof of Proposition~\ref{prop:joining with Chacon} with $m=\Delta_\pi$, and use the same notations as in this proof.
  Since $T$ is ergodic, $R$ is also ergodic, hence the weights $w_i(z)$, $i\in\ZZ$, which are $R$-invariant, are $\rho$-almost everywhere constant. By construction, $w_0>0$, and we claim that 
  for $i\neq0$, $w_i=0$. Indeed, otherwise we would have, for $\rho$-almost all $z$,
  $z=\pi(\varphi(z))=\pi(T^i \varphi(z))$. This would imply that, for $\mu$-almost all $x$, $\pi(x)=\pi(T^ix)$, hence $\pi$ would be constant as $T^i$ is ergodic. This is impossible because $\left(Z,\mathcal{Z},\rho,R\right)$ cannot be reduced to a single point system (since $\rho$ is $\sigma$-finite). 
  
  We conclude that the conditional measure $\mu_z$ is $\rho$-almost everywhere the Dirac mass at $\varphi(z)$. Therefore, $\pi$ is inversible, and its inverse is $\varphi$. 
\end{proof}

\begin{remark}
  It is easily seen that all the results proved in Section~\ref{sec:joinings} are valid for 
  any dynamical system $\left(X,\A,\mu,T\right)$ for which the conclusion of Proposition~\ref{prop:alphamu} holds.
  Concerning Corollary~\ref{cor:law_of_large_numbers}, it is known in fact that Chacon infinite transformation admits a \emph{measurable} law of large numbers: this is a consequence 
  of Theorem~3.3.1 in~\cite{Aaronson}, and the fact that Chacon infinite transformation is rationally ergodic~\cite{Silva_et_al2015}. 
\end{remark}

\begin{appendix}

\section{Product theorem}

\begin{theo}
\label{thm:product}
  Let $X$ and $Y$ be two standard Borel measurable spaces. Let $T\ :\ X\to X$ and $S\ : Y\to Y$ be invertible, bi-measurable transformations. 
  Let $\sigma$ be a $\sigma$-finite measure on $X\times Y$ satisfying 
  \begin{itemize}
    \item there exist $X_0\subset X$ and $Y_0\subset Y$ with $0<\sigma(X_0\times Y_0)<\infty$,
    \item $\sigma$ is $T\times S$-invariant,
    \item the dynamical system $(X\times Y,T\times S,\sigma)$ is conservative and ergodic,
    \item $\Id\times S$ is non-singular with respect to $\sigma$.
  \end{itemize}
  
  Then, $\sigma$ is in fact $\Id\times S$-invariant, and there exist two measures $\mu$ and $\nu$ respectively on $X$ and $Y$, invariant by $T$ and $S$,  such that $\sigma=\mu\otimes\nu$. Moreover, the dynamical systems $(X,\mu, T)$ and $(Y,\nu,S)$ are conservative and ergodic.
  \end{theo}
  
\begin{proof}
  
  Since $\Id\times S$ commutes with $T\times S$, the density 
  \[
      \frac{d (\Id\times S)_*\sigma}{d\sigma}(x,y)
  \]
  is $T\times S$-invariant. Hence, by ergodicity, it is $\sigma$-almost everywhere equal to some constant $c$, $0<c<\infty$.

Set, for each $n\in\ZZ$, $X_{n}\egdef T^{n}X_{0}$, and $Y_{n}\egdef S^{n}Y_{0}$, where $X_0$ and $Y_0$ are given in the assumptions of the theorem.
As $\sigma$ is invariant by $T\times S$, we deduce that, for all $\left(m,n\right)\in\mathbb{Z}^{2}$, 
\[ \sigma\left(X_{n}\times Y_{m}\right) = \sigma\left(X_{0}\times Y_{m-n}\right)
=c^{n-m}\sigma\left(X_{0}\times Y_{0}\right).
\]
Choose two sequences of positive numbers $\left(k_{n}\right)_{n\in\mathbb{Z}}$
and $\left(\ell_{n}\right)_{n\in\mathbb{Z}}$ such that 
\[
 \sum_{\left(n,m\right)\in\mathbb{Z}^{2}}k_{n}\ell_{m}c^{n-m}
 =\left(\sum_{n\in\mathbb{Z}}k_{n}c^{n}\right)\left(\sum_{m\in\mathbb{Z}}\ell_{m}c^{-m}\right)<\infty.                
 \]
Define $f\egdef\sum_{n\in\mathbb{Z}}k_{n}\ind{X_{n}}$ and $g\egdef\sum_{n\in\mathbb{Z}}\ell_{n}\ind{Y_{n}}$.
As $f\otimes g$ is supported on $\cup_{\left(n,m\right)\in\mathbb{Z}^{2}}\left(X_{n}\times Y_{m}\right)$
which contains $\cup_{n\in\mathbb{Z}}\left(X_{n}\times Y_{n}\right)=X\times Y$
mod $\sigma$ (by ergodicity of $T\times S$), we deduce that $f\otimes g>0$
$\sigma$-a.e. Moreover,
\[
\int_{X\times Y}f\otimes g\,d\sigma=\sigma\left(X_{0}\times Y_{0}\right)\left(\sum_{n\in\mathbb{Z}}k_{n}c^{n}\right)\left(\sum_{m\in\mathbb{Z}}\ell_{m}c^{-m}\right)<\infty.
\]
So we can assume that $\int_{X\times Y}f\otimes g\,d\sigma=1$, and we can
define the probability measure $\rho$ whose density with respect to $\sigma$ is equal to $f\otimes g$. We denote its respective projections on $X$ and $Y$ by $\rho_X$ and $\rho_Y$. 

Let us compute the density of $(\Id\times S)_*(\rho)$ with respect to $\rho$. For any measurable non-negative functions $h$ on $X$ and $k$ on $Y$, we have
\begin{align*}
&\int_{X\times Y}h\otimes k\circ\left(\Id\times S\right)\left(x,y\right)d\rho(x,y)\\
 & =\int_{X\times Y}h\left(x\right)k\left(Sy\right)f\left(x\right)g\left(y\right)d\sigma(x,y)\\
 & =c \int_{X\times Y}h\left(x\right)k\left(y\right)f\left(x\right)g\left(S^{-1}y\right)d\sigma(x,y)\\
 & =c \int_{X\times Y}h\left(x\right)k\left(y\right)\frac{g\left(S^{-1}y\right)}{g\left(y\right)}d\rho (x,y).
\end{align*}
This proves that the sought-after density is equal to $c\frac{g\left(S^{-1}y\right)}{g\left(y\right)}$. In particular, it only depends on $y$, and by taking $h=1$ in the above computation, we get that $S$ is non-singular with respect to $\rho_Y$, with the same density. 

Now we wish to prove that the non-singular dynamical system $(Y, \rho_Y, S)$ is ergodic and conservative. Indeed, if $A$ is an $S$-invariant set with $\rho_Y(A)>0$, then $X\times A$ is $T\times S$-invariant with $\rho(X\times A)>0$. By ergodicity of $T\times S$, $\rho(X\times A)=1$ and $\rho_Y(A)=1$. 
In the same vein, if $W$ is a wandering set for $S$, then $X\times W$
is a wandering set for $T\times S$, therefore $\rho_Y\left(W\right)=\rho\left(X\times W\right)=0$,
by conservativity of $T\times S$.

Consider the measure $\nu$ on $Y$ whose density with respect to $\rho_Y$ is equal to $1/g(y)$. It is straightforward to check that the density of $S_*(\nu)$ with respect to $\nu$ is constant equal to $c$. 
We claim that $c=1$.
Indeed, we consider the Maharam extension of $S$ defined on $(Y\times \RR_+^*,\nu\otimes dt)$ by
\[
  \tilde S(y,t)\egdef  (Sy, t/ c) \in Y\times \RR_+^*.
\]
Observe that if $c\neq 1$, $\tilde S$ is totally dissipative. But we know that $(Y,S,\nu)$ is conservative, hence $\tilde S$ is also conservative by Theorem~2 in~\cite{Maharam1964}, and we conclude that $c=1$.
This proves that $\sigma$ is in fact invariant by $\Id\times S$.

The same arguments applied on the first coordinate lead to similar results: If $\mu$ is the measure on $X$ whose density with respect to $\rho_X$ is equal to $1/f(x)$, then $\mu$ is invariant by $T$, and the measure-preserving dynamical system $(X,\mu,T)$ is conservative and ergodic. 

The end of the proof is an application of Lemma~3.1.1 in~\cite{RudolphSilva} to the measure $\rho$: This lemma proves that $\rho$ is the product of its marginals $\rho_X$ and $\rho_Y$, thus $\sigma=\mu\otimes\nu$.
\end{proof}

\end{appendix}

\bibliography{chacon-infinite}

\end{document}